\documentclass{amsart}

\usepackage{amsmath}
\usepackage{amsthm}
\usepackage{amssymb}
\usepackage{hyperref}
\usepackage{tabularx}

\newtheorem{theorem}{Theorem}[section]
\newtheorem{lemma}[theorem]{Lemma}

\theoremstyle{definition}

\theoremstyle{remark}

\newcommand{\sbu}{\sigma^{**}}
\newcommand{\Mod}[1]{\ \left(\mathrm{mod}\ #1\right)}
\newcommand{\seqnum}[1]{\href{https://oeis.org/#1}{\rm \underline{#1}}}

\begin{document}
\title{Determining all biunitary triperfect numbers of a certain form}
\author{Tomohiro Yamada}
\address{Center for Japanese language and culture, Osaka University,
562-8678, 3-5-10, Sembahigashi, Minoo, Osaka, Japan}
\email{tyamada1093@gmail.com}

\date{}
\maketitle
\title{}

\begin{abstract}
We shall show that $2160$ is the only biunitary triperfect number divisible by $27=3^3$.
\end{abstract}

\section{Introduction}\label{intro}
As usual, we let $\sigma(N)$ and $\omega(N)$ denote respectively
the sum of divisors and the number of distinct prime factors of a positive integer $N$.
We call a positive integer $N$ to be perfect if $\sigma(N)=2N$.
It is a well-known unsolved problem whether or not
an odd perfect number exists.
Interest to this problem has produced many analogous concepts
and problems concerning divisors of an integer.
For example, it is also unknown whether or not there exists an odd multiperfect number,
an integer dividing the sum of its divisors.
We call a positive integer $N$ to be $k$-perfect if $\sigma(N)=kN$.
Ordinary perfect numbers are $2$-perfect numbers
and multiperfect numbers are $k$-perfect numbers for some integer $k$.
A $3$-perfect number such as $120$, $672$, and $523776$ are also called to be {\it triperfect},
which are given in \seqnum{A005820}.

On the other hand, some special classes of divisors have also been studied in several papers.
One of them is the class of unitary divisors defined by Cohen \cite{CohE},
who called a divisor $d$ of $N$ to be a {\it unitary divisor} if $\gcd(d, N/d)=1$.
Wall \cite{Wal1} introduced the notion of {\it biunitary divisors},
a divisor $d$ of a positive integer $N$ satisfying $\gcd_1(d, N/d)=1$,
where $\gcd_1(a, b)$ is the greatest common unitary divisor of two positive integers $a$ and $b$.

According to Cohen \cite{CohE} and Wall \cite{Wal1} respectively,
we let $\sigma^*(N)$ and $\sbu(N)$ denote the sum of unitary and biunitary divisors of $N$.
Moreover, we write $d\mid\mid N$ if $d$ is a unitary divisor of $N$.
Hence, for a prime $p$, we have $p^e\mid\mid N$ if $p$ divides $N$ exactly $e$ times.
Replacing $\sigma$ by $\sigma^*$, Subbarao and Warren \cite{SW}
introduced the notion of {\it unitary perfect numbers}
by calling $N$ to be unitary perfect if $\sigma^*(N)=2N$.
They proved that there are no odd unitary perfect number and $6$, $60$, $90$, and $87360$
are the first four unitary perfect numbers.
Nine years later, the fifth unitary perfect number was found by Wall \cite{Wal2},
but no further instance has been found.
Subbarao \cite{Sub} conjectured that there are only finitely many
unitary perfect numbers.

Similarly, a positive integers $N$ is called {\it biunitary perfect} if $\sbu(N)=2N$.
Wall \cite{Wal1} showed that $6$, $60$, and $90$, the first three unitary perfect numbers,
are the only biunitary perfect numbers.

Combining the notion of multiperfect numbers and biunitary perfect numbers,
Hagis \cite{Hag} introduced the notion of biunitary multiperfect numbers,
integers $N$ such that $\sbu(N)=kN$ for some integer $k$ and proved that there is
no odd biunitary multiperfect number.
Smallest instances for $\sbu(N)=kN$ with $k\geq 3$ are $N=120, 672, 2160, \ldots$ with $k=3$ 
and $N=30240, 1028160, 6168960, \ldots$ with $k=4$.
All biunitary multiperfect numbers below $4.66\times 10^{12}$ as well as many larger ones
are given in \seqnum{A189000}.

Now we can call an integer $N$ satisfying $\sbu(N)=kN$ to be biunitary $k$-perfect
and biunitary $3$-perfect numbers to be biunitary triperfect.
Haukkanen and Sitaramaiah \cite{HS} determined all biunitary triperfect numbers $N$
such that $2^a\mid\mid N$ with $1\leq a\leq 6$ or $a=8$,
and such ones with $a=7$ under several conditions.
In this paper, we shall determine all biunitary triperfect numbers divisible by $27=3^3$.

\begin{theorem}\label{th1}
There exists the only one biunitary triperfect number $N=2160$ divisible by $3^3$.
\end{theorem}

Our proof is completely elementary.
If $\sbu(N)=3N$ with a factorization $N=\prod_i p_i^{e_i}$, then
a prime $p\neq 3$ dividing $\sbu(p_i^{e_i})$ must divide $N$ since $\sbu(p_i^{e_i})\mid \sbu(N)=3N$.
We see that $\sbu(2^e 3^f)/(2^e 3^f)$ tends to three as $e$ and $f$ grows
and for $e$ and $f$ large, $\sbu(2^e)$ and $\sbu(3^f)$ produce new prime factors of $N$.
In many cases, this causes $\sbu(N)/N>3$, which is a contradiction.
In other cases, we are led to a contradiction that $p^{f+1}\mid N$ under the assumption that
$p$ divides $N$ exactly $f$ times or $\sbu(N)/N<3$.

Based on our theorem and known biunitary multiperfect numbers, we can pose the following problems:
\begin{itemize}
\item For each integer $k\geq 3$, are there infinitely or only finitely many integers $N$ for which $\sbu(N)=kN$?
\item For any prime power $p^g\neq 1$, does there exist at least one $N$
for which $\sbu(N)=kN$ and $p^g\mid\mid N$?
\item For any integer $d>2$, does there exist at least one integer $N$
for which $\sbu(N)=kN$ and $\gcd(d, N)=1$?
\end{itemize}

\section{Preliminary Lemmas}\label{lemmas}
In this section, we shall give several preliminary lemmas concerning the sum of
biunitary divisors used to prove our main theorems.

Before all, we introduce two basic facts from \cite{Wal1}.
The sum of biunitary divisors function $\sbu$ is multiplicative.
Moreover, if $p$ is a prime and $e$ is a positive integer, then
\begin{equation}
\sbu(p^e)= \begin{cases}
p^e+p^{e-1}+\cdots +1=\frac{p^{e+1}-1}{p-1}, & \text{if $e$ is odd;} \\
\frac{p^{e+1}-1}{p-1}-p^{e/2}=\frac{(p^{e/2}-1)(p^{e/2+1}+1)}{p-1}, & \text{if $e$ is even.}
\end{cases}
\end{equation}
We note that, putting $e=2s-1-\delta$ with $\delta\in\{0, 1\}$, this can be represented by the single formula
\begin{equation}\label{eq21}
\sbu(p^e)=\frac{(p^{s-\delta}-1)(p^s+1)}{p-1}.
\end{equation}

From these facts, we can deduce the following lemmas almost immediately.

\begin{lemma}\label{a}
Let $n$ be a positive integer.
Then, $\sbu(n)$ is odd if and only if $n$ is a power of $2$ (including $1$).
More exactly, $\sbu(n)$ is divisible by $2$ at least $\omega(n)$ times if $n$ is odd
and at least $\omega(n)-1$ times if $n$ is even.
\end{lemma}
\begin{proof}
Whether $e$ is even or odd, $\sbu(p^e)$ is odd if and only if $p=2$
by \eqref{eq21}.
Factoring $n=2^e \prod_{i=1}^r p_i^{e_i}$
into distinct odd primes $p_1, p_2, \ldots, p_r$ with $e\geq 0$ and $e_1, e_2, \ldots, e_r>0$,
each $\sbu(p_i^{e_i})$ is even.
Hence, $\sbu(n)=\sbu(2^e)\prod_{i=1}^r \sbu(p_i^{e_i})$ is divisible by $2$ at least $r$ times,
where $r=\omega(n)$ if $n$ is odd and $\omega(n)-1$ if $n$ is even.
\end{proof}

\begin{lemma}\label{b}
For any prime $p$ and any positive integer $e$, $\sbu(p^e)/p^e\geq 1+1/p^2$.
Moreover, $\sbu(p^e)/p^e\geq 1+1/p$ unless $e=2$
and $\sbu(p^e)/p^e\geq (1+1/p)(1+1/p^3)$ if $e\geq 3$.
More generally, for any positive integers $m$ and $e\geq 2m-1$,
we have $\sbu(p^e)/p^e\geq \sbu(p^{2m})/p^{2m}$
and, unless $e=2m$, $\sbu(p^e)/p^e\geq 1+1/p+\cdots +1/p^m$.
\end{lemma}
\begin{proof}
If $e\geq 2m-1$ and $e$ is odd, then $p^e, p^{e-1}, \ldots, p, 1$ are biunitary divisors of $p^e$.
If $e>2m$ and $e$ is even, then $p^e, p^{e-1}, \ldots, p^{e-m}$ are biunitary divisors of $p^e$ since $e-m>e/2$.
Hence, if $e\geq 2m-1$ and $e\neq 2m$, then $\sbu(p^e)=p^e+p^{e-1}+\cdots +1>p^e+\cdots +p^{e-m}=p^e(1+1/p+\cdots +1/p^m)$.
Since $\sbu(p^{2m})/p^{2m}<1+1/p+\cdots +1/p^m$, $\sbu(p^e)/p^e$ with $e\geq 2m-1$ takes its minimum value at $e=2m$.
\end{proof}

Now we shall quote the following lemma of Bang \cite{Ban}, which has been rediscovered
(and extended into numbers of the form $a^n-b^n$)
by many authors such as Zsigmondy \cite{Zsi}, Dickson \cite{Dic} and Kanold \cite{Kan}.
See also Theorem 6.4A.1 in Shapiro \cite{Sha}.

\begin{lemma}\label{lm21}
If $a\geq 2$ and $n\geq 2$ are integers, then $a^n-1$ has a prime factor
which does not divide $a^m-1$ for any $m<n$, unless $(a, n)=(2, 1), (2, 6)$
or $n=2$ and $a+1$ is a power of $2$.
Furthermore, such a prime factor must be congruent to $1$ modulo $n$.
\end{lemma}

As a corollary, we obtain the following lemma:

\begin{lemma}\label{c}
If $a\geq 2$ and $n\geq 1$ are integers, then $a^n+1$ has a prime factor
which does not divide $a^m+1$ for any $m<n$ unless $(a, n)=(2, 3)$.
Furthermore, such a prime factor must be congruent to $1$ modulo $2n$.
\end{lemma}
\begin{proof}
By Lemma \ref{lm21}, $a^{2n}-1=(a^n-1)(a^n+1)$ has a prime factor $p$ which does not divide $a^m-1$ for any $m<2n$.
Since $p$ does not divide $a^n-1$, $p$ must divide $a^n+1$.
On the other hand, for $m<n$, $p$ cannot divide $a^m+1$ since $p$ does not divide $a^{2m}-1$.
Finally, such a prime factor $p$ must be congruent to $1$ modulo $2n$.
\end{proof}

Now we prove that $\sbu(2^e)$ and $\sbu(3^f)$ must produce a new prime factor which is not very large.
\begin{lemma}\label{lm22}
Let $f$ be a positive integer and write $f=2t-1-\eta$ with $\eta\in \{0, 1\}$ and $t$ integers.
If $f\geq 5$, then at least one of the following statements holds.
\begin{itemize}
\item[(A)] $\eta=0$ and $\sbu(3^f)$ has a prime factor $p_1$ with $5<p_1\leq (3^t-1)/2$,
\item[(B)] $p_1=5\mid\sbu(3^f)$ and $f\equiv 2\Mod{4}$ or $f=7, 8$,
\item[(C)] $\sbu(3^f)$ has an odd prime factor $p_1$ with $5<p_1\leq\sqrt{(3^{t-\eta}-1)/2}$, or
\item[(D)] $4$ divides $t$, $\eta=1$, and $p_1=(3^{t-1}-1)/2$ is prime.
\end{itemize}
\end{lemma}

\begin{proof}
We begin by observing that $\sbu(3^f)=(3^t+1)(3^{t-\eta}-1)/2$ from \eqref{eq21}.
We put $t=2^h v$ with $v$ odd.

If $\eta=0$ and $f\neq 7$, then, we see that $t=3$ or $t\geq 5$.
Hence, Lemma \ref{lm21} implies that $(3^t-1)/2$ has a prime factor greater than $5$.
Putting $p_1$ to be the smallest one among such prime factors,
we have $5<p_1\leq (3^t-1)/2$ and (A) holds.
If $f=7$, then $\sbu(3^f)=(3^8-1)/2=2^4\times 5\times 41$ and (B) holds.

Henceforth we assume that $\eta=1$.
If $t$ is odd and $t\neq 5, 9$, then, $(t-1)/2=3$ or $(t-1)/2\geq 5$ and, like above,
$(3^{(t-1)/2}-1)/2$ has a prime factor greater than $5$.
Take the smallest $p_1$ among such prime factors.
Then, we have
$p_1\leq (3^{(t-1)/2}-1)/2<\sqrt{(3^{t-1}-1)/2}$ and
$p_1\mid (3^{(t-1)/2}-1)/2\mid (3^{t-1}-1)/2\mid \sbu(3^f)$.
Hence, (C) holds.

If $t\equiv 2\Mod{4}$, then $5\mid (3^t+1)\mid \sbu(3^f)$ and (B) holds.

If $4\mid t$, then, putting $p_1$ to be the smallest prime factor of $(3^{t-1}-1)/2$,
clearly $p_1=(3^{t-1}-1)/2$ is prime or $p_1\leq \sqrt{(3^{t-1}-1)/2}$.
Since $t-1$ is odd, we can never have $p_1=2$ or $p_1=5$.
Thus we see that $p_1=(3^{t-1}-1)/2$ is prime or $5<p_1\leq \sqrt{(3^{t-1}-1)/2}$,
implying (D) and (C) respectively.

If $t=9$, then $7\mid (3^3+1)\mid (3^9+1)\mid \sbu(3^f)$ and (C) holds.
Finally, if $t=5$, then $\sbu(3^f)=\sbu(3^8)=2^5\times 5\times 61$ and (B) holds.
\end{proof}

\begin{lemma}\label{lm23}
Let $e\geq 6$ be an integer and write $e=2s-1-\delta$ with $\delta\in \{0, 1\}$ and $s$ integers.
If $e\neq 8, 12$, then at least one of the following statements holds.
\begin{itemize}
\item[(a)] $\delta=0$ and $\sbu(2^e)$ has at least two prime factors $q_1$ and $q_2$ with $5<q_1, q_2\leq 2^s-1$,
\item[(b)] $p_2=5\mid\sbu(2^e)$.
\item[(c)] $\sbu(2^e)$ has at least two prime factors $q_1$ and $q_2$ each of which satisfies
either $q_i=2^{s-1}-1$, $q_i=2^s+1$, or $5<q_i\leq\sqrt{2^s-3}$.
Moreover, if $q_i=2^{s-1}-1$ or $q_i=2^s+1$ for $i=1$ or $2$, then $4$ divides $s$ and $\delta=1$.
\end{itemize}
\end{lemma}

\begin{proof}
We begin by observing that $\sbu(2^e)=(2^s+1)(2^{s-\eta}-1)$ from \eqref{eq21}.
We put $s=2^g m$ with $m$ odd.

If $\delta=0$, $m>1$, and $e\neq 11, 23$, then $s\neq 6, 12$ and
Lemma \ref{lm21} yields that $2^s-1$ has a prime factor $q_1\equiv 1\Mod{s}$
and $2^{2s}-1$ has a prime factor $q_2\equiv 1\Mod{2s}$ not dividing $2^s-1$ or $2^{2^{g+1}}-1$.
Clearly, we see that $q_1\neq q_2$ and $q_2$ divides $(2^s+1)/(2^{2^g}+1)$.
Moreover, since $s\geq 3$ is odd, we see that $q_1\equiv q_2\equiv 1\Mod{2s}$ and $q_2$ divides $(2^s+1)/(2^{2^g}+1)<2^s-1$.
Then, we have $5<q_1, q_2\leq 2^s-1$ and (a) holds.
If $e=11$ or $e=23$, then $7\times 13\mid (2^{12}-1)\mid\sbu(2^f)$ and therefore we can take
$(q_1, q_2)=(7, 13)$, which yields (a) again.

If $\delta=0$, $m=1$ and $s=2^g\geq 16$, then
we work as above but with $2^{s/2}-1$ and $2^s-1$ instead of $2^s-1$ and $2^{2s}-1$ respectively.
Now we can take two prime factors $q_1, q_2$ of $2^s-1$ with $5<q_1, q_2\leq 2^s-1$,
which yields (a).
If $e=7$ or $e=15$, then $5\mid (2^8-1)\mid\sbu(2^e)$ and (b) holds.

Henceforth, assume that $\delta=1$.
If $s$ is odd and $e\neq 8, 12, 16, 24$, then $s\neq 5, 7, 9, 13$ and we see from Lemma \ref{lm21} that
we can take a prime factor $q_1$ of $2^{(s-1)/2}-1$ such that $q_1\equiv 1\Mod{(s-1)/2}$
and a prime factor $q_2$ of $2^{(s-1)/2}+1$ such that $q_2\equiv 1\Mod{(s-1)/2}$.
Since $2^{(s-1)/2}-1$ and $2^{(s-1)/2}+1$ are relatively prime, we have $q_1\neq q_2$.
Now $5<q_1, q_2\leq 2^{(s-1)/2}+1<\sqrt{2^s-3}$ and (c) holds.

If $g=1$, then $5=(2^2+1)\mid (2^s+1)\mid \sigma(2^e)$.

Assume that $g\geq 2$.
Let $q_1$ and $q_2$ be the smallest prime factors of $2^{s-1}-1$ and $2^s+1$ respectively.
Since $s-1$ is odd, we cannot have $q_1=3$ or $q_1=5$.
Similarly, since $4$ divides $s$, we cannot have $q_2=3$ or $q_2=5$.

Since $2^s$ is square and $2^{s-1}-1\equiv 3\Mod{4}$, neither of $2^s+1$, $2^s-1$, nor $2^{s-1}-1$ can be square.
Hence, we see that either $q_1=2^{s-1}-1$ is prime or $q_1\leq\sqrt{2^s-3}$.
Similarly, $q_2=2^s+1$ is prime or $q_2\leq\sqrt{2^s-3}$.
Hence, we see that (c) holds.

If $e=24$, then we can take $(q_1, q_2)=(7, 13)$ since $7\times 13\mid (2^{12}-1)\mid\sbu(2^{24})$ to obtain (c).
Finally, if $e=16$, then we can take $(q_1, q_2)=(17, 19)$ since $17\times 19\mid (2^8-1)(2^9+1)=\sbu(2^{16})$
to obtain (c) observing that $p_2=19<\sqrt{2^9-3}$.
\end{proof}

We also use the following miscellaneous divisibility results.
\begin{lemma}\label{lm24}
(I) For any prime $p$ and $g\geq 2$, $\sbu(p^g)$ has a prime factor $\geq 5$.

(II) For any odd prime $p$ and $g\geq 4$, $\sbu(p^g)$ has a prime factor $\geq 7$.

(III) If $f\geq 18$ or $f=13, 14$, then $\sbu(3^f)$ has at least two distinct prime factors $\geq 127$.

(IV) If $g\geq 2$, then $\sbu(13^g)$ has at least two distinct prime factors
in addition to $2$, $3$, $41$, and $547$.

(V) For any $g\geq 1$, $\sbu(41^g)$ has a prime factor in addition to $2$, $3$, $5$, and $13$.
\end{lemma}

\begin{proof}
(I) We write $g=2u-1-\gamma$ with $\gamma\in \{0, 1\}$.
For $g\geq 2$, we have $u\geq 2$ and, by Lemma \ref{c}, $p^u+1$ has a one prime factor $\geq 5$.

(II) We write $g=2u-1-\gamma$ with $\gamma\in \{0, 1\}$.
For $g\geq 4$, we have $u\geq 3$ and, by Lemma \ref{c}, $p^u+1$ has a one prime factor $\geq 7$.

(III) By Lemma \ref{lm21}, $(3^s-1)/2$ has at least one prime factor $\geq 127$ for $s\geq 126$
and, by Lemma \ref{c}, $3^t+1$ has at least one prime factor $\geq 127$ for $t\geq 126$.
Now, we can confirm that $(3^s-1)/2$ has at least one prime factor $\geq 127$ for $s\geq 7$
and $3^t+1$ has at least one prime factor $\geq 127$ for $t\geq 10$ (see, for example, \url{https://homes.cerias.purdue.edu/~ssw/cun/pmain524.txt}).
Hence, putting $f=2t-1-\eta$ with $\eta\in \{0, 1\}$,
$\sbu(3^f)=(3^t+1)(3^{t-\eta}-1)/2$ has at least two distinct prime factors $\geq 127$ for $f\geq 18$.
For $f=13, 14$, (I) can be confirmed by noting that $\sbu(3^{13})=2^2\times 547\times 1093$ and
$\sbu(3^{14})=2\times 17\times 193\times 547$ .

(IV) With the aid of Lemmas \ref{lm21} and \ref{c}, we see that
$(13^s-1)/12$ has at least one prime factor $\geq 11$ for $s\geq 3$
and $13^t+1$ has at least one prime factor $\geq 11$ for $t\geq 2$
(these can be confirmed by direct calculation for $3\leq s\leq 6$ and $2\leq t\leq 5$).
If $41$ divides $13^t+1$, then $t\equiv 20\Mod{40}$ and $14281$ also divides $13^t+1$.
Similarly, if $41$ divides $(13^s-1)/12$, then $14281$ also divides $(13^s-1)/12$.
If $547$ divides $(13^u-1)/12$, then $21\mid s$ and $61$ also divides $(13^s-1)/12$.
Since $13^{21}\equiv 1\Mod{547}$, $547$ can never divide $13^t+1$.
Hence, for $g\geq 5$, putting $g=2u-1-\gamma$ with $\gamma\in \{0, 1\}$,
each of $13^u+1$ and $(13^{u-\gamma}-1)/12$ has at least one prime factor $p_1$ and $p_2$ respectively
both other than $2$, $3$, $7$, $41$, or $547$.
If $p_1=p_2$, then $p_1$ must divide $13^\gamma+1$ and therefore $p_1=2$ or $p_1=7$, which is a contradiction.
Hence, $\sbu(13^g)=(13^u+1)(13^{u-\gamma-1}-1)/12$ has at least two prime factors besides
$2$, $3$, $41$, and $547$ for $g\geq 5$.
For $2\leq g\leq 4$, (IV) can be easily confirmed.

(V) For $p=41$, from Lemma \ref{c}, we see that $41^u+1$ has a prime factor greater than $5$ for $u\geq 3$.
If $13$ divides $41^u+1$, then $u\equiv 6\Mod{12}$ and $29$ also divides $41^u+1$.
Hence, for $g\geq 2$, putting $g=2u-1-\gamma$ with $\gamma\in \{0, 1\}$,
$41^u+1$ has a prime factor in addition to $2$, $3$, $5$, and $13$
and so does $\sbu(41^g)$.
For $g=1$, (V) is clear since $\sbu(41)=41+1=2\times 3\times 7$.
\end{proof}

\section{Proof of the Theorem}

Assume that $N$ is a biunitary triperfect number.
We put integers $e$ and $f$ by $2^e\mid\mid N$ and $3^f\mid\mid N$
and write $e=2s-1-\delta$ with $\delta\in \{0, 1\}$ and $f=2t-1-\eta$ with $\eta\in \{0, 1\}$.

In this section, once we write $p_i$ for a prime factor of $N$ with an index $i$,
$e_i$ denotes the exponent of $p_i$ dividing $N$.
Clearly, $\sbu(p_i^{e_i})$ divides $N$ for each $i$.

We begin with small $e$ and $f$.
Although Haukkanen and Sitaramaiah \cite{HS} have proved that $e\geq 7$ whether $3^3\mid N$ or not,
we shall give a proof for small $e$ in the view of self-containedness.

\begin{lemma}\label{lm31}
We must have $e\geq 4$.
\end{lemma}
\begin{proof}
We cannot have $e=0$ as Hagis \cite{Hag} has shown.
Indeed, if $N>1$ is odd, then, by Lemma \ref{a}, $\sbu(N)$ must be even and $\sbu(N)\neq 3N$.
If $e=1$, then $N$ can have no odd prime factor other than $3$ and therefore
we must have $\sbu(3^f)=2^e$.
By Lemma \ref{lm21}, we must have $f=1$, contrary to the assumption that $f\geq 3$.

If $e=2$, then $p_1=5=\sbu(2^2)\mid \sbu(N)=3N$ and therefore $5$ must divide $N$.
From Lemma \ref{a}, we see that $N$ has no further prime factor.
Hence, $\sbu(3^f)$ can have no prime factor other than $2$ or $5$.
By Lemma \ref{lm24} (II), we must have $f=3$.
Now $e_1\geq 2$ since $5^2\mid \sbu(2^2 3^3)\mid 3N$.
However, from Lemma \ref{lm24} (I), we see that $p_2\mid \sbu(5^{e_1})\mid \sbu(N)=3N$ for some $p_2>5$
and $2^3\mid N$ by Lemma \ref{a}, which is a contradiction.

If $e=3$, then $p_1=5$ must divide $N$.
If $e_1\neq 2$, then, with the aid of Lemma \ref{b}, we have
\begin{equation}
\frac{\sbu(N)}{N}\geq \frac{\sbu(2^3 3^4 5)}{2^3 3^4 5}=\frac{28}{9}>3,
\end{equation}
contrary to the assumption that $\sbu(N)=3N$.
Hence, we must have $e_1=2$ and $p_2=13$ must divide $N$ since $13\mid \sbu(5^2)\mid \sbu(N)=3N$.
With the aid of Lemma \ref{lm24} (II), we see that $\sbu(13^{e_2})$ has an odd prime factor $p_3$
in addition to $3$, $5$, and $13$.
By Lemma \ref{a}, we must have $2^4\mid\sbu(3^f 5^2 13^{e_2} p_3^{e_3})\mid 3N$ and $2^4\mid N$,
which is a contradiction.
Thus, we cannot have $e=3$.
\end{proof}

\begin{lemma}\label{lm32}
We cannot have $e=5$.
\end{lemma}
\begin{proof}
Assume that $e=5$.
Then, $\sbu(2^5)=3^2\times 7$ and $p_1=7$ must divide $N$.
If $e_1\neq 2$, then
\begin{equation}
\frac{\sbu(N)}{N}\geq \frac{\sbu(2^5 3^4 7)}{2^5 3^4 7}=\frac{28}{9}>3
\end{equation}
and if $f\neq 4$ and $e_1=2$, then $5\mid\sbu(p_1^{e_1})=7^2+1$ and
\begin{equation}
\frac{\sbu(N)}{N}\geq \frac{\sbu(2^5 3^6 7^2 5^2)}{2^5 3^6 7^2 5^2}=\frac{6929}{2268}>3,
\end{equation}
which are both contradictions.
If $f=4$ and $e_1=2$, then $p_2=5$ divides $N$ and
$2^6\mid \sbu(3^4 7^2 5^{e_2})\mid 3N$, which is a contradiction.
\end{proof}

\begin{lemma}\label{lm33}
We cannot have $e=6$.
\end{lemma}
\begin{proof}
Assume that $e=6$.
We observe that $p_1=7$ and $p_2=17$ must divide $N$ since $7\times 17=\sbu(2^6)\mid \sbu(N)=3N$.
If $e_1\neq 2$ and $e_2\neq 2$, then
\begin{equation}
\frac{\sbu(N)}{N}\geq \frac{\sbu(2^6 3^4 7\times 17)}{2^6 3^4 7\times 17}=\frac{28}{9}>3,
\end{equation}
which is a contradiction.
If $e_1\neq 2$ and $e_2=2$, then, since $5$ divides $17^2+1$, $p_3=5$ must divide $N$ and
\begin{equation}
\frac{\sbu(N)}{N}\geq \frac{\sbu(2^6 3^4 5^2 7)}{2^6 3^4 5^2 7}=\frac{6188}{2025}>3,
\end{equation}
which is a contradiction again.

Now we must have $e_1=2$.
We see that $p_3=5$ and $e_3\geq 2$ since $5^2\mid (7^2+1)=\sbu(7^2)$.
We observe that we cannot have $f=4$ since $2^7\mid \sbu(3^4 7^2 17^{e_2} 5^{e_3})$.
If $e_3>2$, then
\begin{equation}
\frac{\sbu(N)}{N}>\frac{\sbu(2^6 3^6 7^2 5)}{2^6 3^6 7^2 5}=\frac{45305}{13608}>3,
\end{equation}
which is a contradiction.
Hence, $e_3=2$ and $p_4=13$ must divide $N$.
However, this is impossible.
Indeed, if $e_4=2$, then $5^3\mid\sbu(7^2 13^2)\mid 3N$, contrary to $e_3=2$,
and if $e_4\neq 2$, then
\begin{equation}
\frac{\sbu(N)}{N}\geq \frac{\sbu(2^6 3^6 7^2 5^2 13)}{2^6 3^6 7^2 5^2 13}=\frac{9061}{2916}>3,
\end{equation}
which is a contradiction.
\end{proof}

\begin{lemma}\label{lm34}
We cannot have $e\geq 7$ and $f=3$.
\end{lemma}
\begin{proof}
Assume that $e\geq 7$ and $f=3$.
Clearly $p_1=5$ must divide $N$ since $5\mid \sbu(3^3)\mid 3N$.
If $e_1\neq 2$, then
\begin{equation}
\frac{\sbu(N)}{N}\geq \frac{\sbu(2^8 3^3 5)}{2^8 3^3 5}=\frac{55}{16}>3,
\end{equation}
which is a contradiction.

If $e_1=2$, then, since $\sbu(5^{e_1})=2\times 13$, $p_2=13$ must divide $N$.
We must have $e_2=2$, $p_3=17$, and $e_3=2$ since otherwise
\begin{equation}
\frac{\sbu(N)}{N}\geq \frac{\sbu(2^8 3^3 5^2 13)}{2^8 3^3 5^2 13}=\frac{77}{24}>3
\end{equation}
or
\begin{equation}
\frac{\sbu(N)}{N}\geq \frac{\sbu(2^8 3^3 5^2 13^2 17)}{2^8 3^3 5^2 13^2 17}=\frac{165}{52}>3,
\end{equation}
which is a contradiction.
But then, we must have $5^3\mid\sbu(3^3 13^2 17^2)\mid 3N$, contrary to $e_1=2$.
\end{proof}

\begin{lemma}\label{lm35}
We cannot have $e\geq 7$ and $f=4$.
\end{lemma}
\begin{proof}
Assume that $e\geq 7$ and $f=4$.
Clearly $p_1=7$ must divide $N$ since $7\mid \sbu(3^4)\mid 3N$.
If $e_1\neq 2$, then
\begin{equation}
\frac{\sbu(N)}{N}\geq \frac{\sbu(2^8 3^4 7)}{2^8 3^4 7}=\frac{55}{18}>3,
\end{equation}
which is a contradiction.

If $e_1=2$, then, since $\sbu(7^{e_1})=2\times 5^2$, $p_2=5$ must divide $N$.
We must have $e_2=2$, $p_3=13$, and $e_3=2$, since otherwise
\begin{equation}
\frac{\sbu(N)}{N}\geq \frac{\sbu(2^8 3^4 7^2 5)}{2^8 3^4 7^2 5}=\frac{275}{84}>3
\end{equation}
or
\begin{equation}
\frac{\sbu(N)}{N}\geq \frac{\sbu(2^8 3^4 7^2 5^2 13)}{2^8 3^4 7^2 5^2 13}=\frac{55}{18}>3,
\end{equation}
which is a contradiction.
Hence, we must have $5^3\mid N$ since $5^3\mid\sbu(7^2 13^2)$, which is also a contradiction.
\end{proof}

Now we prove remaining cases.
Assume that $f\geq 5$, $e\geq 7$ and $e\neq 8, 12$.
Then, a prime $p_1$ taken from Lemma \ref{lm22} must divide $N$ since $p_1\mid\sbu(3^f)\mid 3N$ and $p_1\neq 3$.
If (b) in Lemma \ref{lm23} holds, then $p_2=5$ must divide $N$ and we have $p_1>5=p_2$.
If (a) or (c) in Lemma \ref{lm23} holds, then we can choose a prime $p_2\in\{q_1, q_2\}$ not equal to $p_1$.
Like above, $p_2$ must divide $N$.
We see that if (c) in Lemma \ref{lm23} holds, then we have either
(c$^\prime$) $5<p_2\leq\sqrt{2^s-3}$ or (d) $4\mid s$, $\delta=1$, and either of
$p_2=2^{s-1}-1$ or $p_2=2^s+1$ is prime.
We put $N_1=2^e p_2^{e_2}$ if (a), (b), or (c$^\prime$) holds or (d) holds with $e_2\neq 2$.
It is clear that $N_1\mid\mid N$.
If (d) holds with $e_2=2$, then, we observe that $p_i\equiv 2\Mod{5}$ and $5\mid (p_i^2+1)\mid\sbu(N)=3N$
and thus we see that $N_1=2^e 5^{e_3}\mid\mid N$.
Similarly, we put $N_2=3^f p_1^{e_1}$ if (A), (B), or (C) holds or (D) holds with $e_2\neq 2$
and $N_2=3^f 5^{e_3}$ if (D) holds with $e_2=2$ to see that $N_2\mid\mid N$.

We observe that $\sbu(N_1)/N_1>2$ and $\sbu(N_2)/N_2>3/2$.
Indeed, if (a) holds, then
\begin{equation}
\frac{\sbu(N_1)}{N_1}
\geq \frac{(2^{e+1}-1)(p_2^2+1)}{2^e p_2^2}
\geq \frac{(2^{2s}-1)((2^s-1)^2+1)}{2^{2s-1}(2^s-1)^2}>2.
\end{equation}
If (c$^\prime$) holds, then
\begin{equation}
\frac{\sbu(N_1)}{N_1}
\geq \frac{(2^{s-1}-1)(2^s+1)(p_2^2+1)}{2^{2s-2} p_2^2}
\geq \frac{(2^{s-1}-1)^2(2^s-2)}{2^{2s-3}(2^s-3)}>2.
\end{equation}
If (d) holds with $e_2\neq 2$, then
\begin{equation}
\frac{\sbu(N_1)}{N_1}
\geq \frac{(2^{s-1}-1)(2^s+1)(p_2+1)}{2^{2s-2} p_2}
\geq \frac{(2^{s-1}-1)(2^s+2)}{2^{2s-2}}>2.
\end{equation}
If (b) holds, then
\begin{equation}
\frac{\sbu(N_1)}{N_1}
\geq \frac{(2^{s-\eta}-1)(2^s+1)(5^{e_2}+1)}{2^{2s-1-\eta} 5^{e_2}}
\geq \frac{26(2^{s-\eta}-1)(2^s+1)}{25\times 2^{2s-1-\eta}}>2,
\end{equation}
which also holds with $e_2$ replaced by $e_3$ in the case (d) with $e_2=2$.
Thus, we see that $\sbu(N_1)/N_1>2$ in any case.

Similarly, if (A) holds, then
\begin{equation}
\frac{\sbu(N_2)}{N_2}
\geq \frac{(3^{f+1}-1)(p_1^2+1)}{2\times 3^f p_1^2}
\geq \frac{(3^{2t}-1)((3^t-1)^2+4)}{2\times 3^{2t-1}(3^t-1)^2}>\frac{3}{2}.
\end{equation}
If (C) holds, then
\begin{equation}
\frac{\sbu(N_2)}{N_2}
\geq \frac{(3^{t-\eta}-1)(3^t+1)(p_1^2+1)}{2\times 3^{2t-1-\eta} p_1^2}
\geq \frac{(3^{t-\eta}-1)(3^t+1)(3^{t-\eta}+1)}{2\times 3^{2t-1-\eta}(3^{t-\eta}-1)}
>\frac{3}{2}.
\end{equation}
If (D) holds with $e_1\neq 2$, then
\begin{equation}
\frac{\sbu(N_2)}{N_2}
\geq \frac{(3^{t-1}-1)(3^t+1)(p_1+1)}{2\times 3^{2t-2} p_1}
\geq \frac{(3^{t-1}-1)(3^t+1)(3^{t-1}+1)}{2\times 3^{2t-2}(3^{t-1}-1)}
>\frac{3}{2}.
\end{equation}
If (B) holds, then
\begin{equation}
\frac{\sbu(N_2)}{N_2}
\geq \frac{(3^{t-\eta}-1)(3^t+1)(5^{e_1}+1)}{2\times 3^{2t-1-\eta} 5^{e_1}}
\geq \frac{26(3^{t-1}-1)(3^t+1)}{2\times 25\times 3^{2t-1-\eta}}
>\frac{3}{2},
\end{equation}
which also holds with $e_1$ replaced by $e_3$ in the case (D) with $e_1=2$.
Thus, we see that $\sbu(N_2)/N_2>3/2$ in any case.

If (a) or (c$^\prime$) holds, then $\gcd(N_1, N_2)=1$ and
\begin{equation}\label{eq51}
\frac{\sbu(N)}{N}\geq \frac{\sbu(N_1)}{N_1}\times \frac{\sbu(N_2)}{N_2}>3,
\end{equation}
which is a contradiction.
If (b) or (d) holds and (A) or (C) also holds, then, we have $\gcd(N_1, N_2)=1$ and a contradiction
\eqref{eq51} again.

Now we settle two cases (i) (b) or (d) holds and (B) or (D) also holds and (ii) $e=8$.
In both cases, $5$ must divide $N$ since $5$ divides $\sbu(2^8)=3^2\times 5\times 11$.
We rewrite $p_1=5$.
If $e_1\neq 2$, then
\begin{equation}
\frac{\sbu(N)}{N}\geq \frac{\sbu(2^8 3^6 5)}{2^8 3^6 5}=\frac{5863}{1728}>3,
\end{equation}
which is a contradiction.
Hence, $e_1=2$.

Now, rewriting $p_2=13$, $p_2$ must divide $N$.
Hence, if $e\neq 8$, then
\begin{equation}
\frac{\sbu(N)}{N}\geq \frac{\sbu(2^{10} 3^6 5^2 13^2)}{2^{10} 3^6 5^2 13^2}>3
\end{equation}
and, if $e=8$ and $f\neq 6$, then
\begin{equation}
\frac{\sbu(N)}{N}\geq \frac{\sbu(2^8 3^8 5^2 13^2)}{2^8 3^8 5^2 13^2}>3
\end{equation}
and we have a contradiction.
Thus, we must have $e=8$ and $f=6$.
Since $f=6$, $p_3=7$ must divide $N$.
If $e_3=2$, then $5^3\mid \sbu(2^8 7^2)\mid 3N$, which is incompatible with $e_1=2$.
If $e_3\neq 2$, then
\begin{equation}
\frac{\sbu(N)}{N}\geq \frac{\sbu(2^8 3^6 7)}{2^8 3^6 7}=\frac{29315}{9072}>3,
\end{equation}
a contradiction again.

Now we assume that $e=12$.
Then, $p_1=7$ must divide $N$ since $7\mid (2^6-1)\mid \sbu(2^{12})$.
If $e_1\neq 2$, then 
\begin{equation}
\frac{\sbu(N)}{N}\geq \frac{\sbu(2^{12} 3^6 7)}{2^{12} 3^6 7}=\frac{22919}{6912}>3,
\end{equation}
and, if $e_1=2$, then $5$ divides $N$ and
\begin{equation}
\frac{\sbu(N)}{N}\geq \frac{\sbu(2^{12} 3^6 7^2 5^2)}{2^{12} 3^6 7^2 5^2}=\frac{297947}{96768}>3.
\end{equation}
Thus, we have a contradiction.

Now the only remaining case is the case $e=4$.
We shall prove that if $e=4$, then $N=2160$ to complete the proof of Theorem \ref{th1}.
We begin by showing that if $N\neq 2160$, then $f\geq 5$ and $5$ cannot divide $N$.

\begin{lemma}\label{lm51}
If $e=4$ and $f\leq 4$, then $N=2160$.
\end{lemma}

\begin{proof}
If $f=4$, then $\sbu(3^f)=2^4 7$ and therefore $p_1=7$ must divide $N$.
and $2^5\mid \sbu(3^f 7^{e_1})\mid 3N$, contrary to the assumption that $e=4$.

If $f=3$, then $\sbu(3^f)=2^3 5$ and $p_1=5$ must divide $N$.
If $e_1=1$, then $\sbu(2^4 3^3 5)=2^4 3^4 5$ and therefore $N=2^4 3^3 5$.
If $e_1=2$, then $p_1^2+1=2\times 13$ and $p_2=13$ must divide $N$,
which is impossible since $2^5\mid \sbu(3^3 5^2 13^{e_2})$.
If $e_1\geq 3$, then
\begin{equation}
\frac{\sbu(N)}{N}>\frac{\sbu(2^4 3^3 5)}{2^4 3^3 5}=3,
\end{equation}
which is a contradiction.
Hence, if $f=3$, then we have $N=2^4 3^3 5$.
\end{proof}

Now we must have $f\geq 5$ if $N\neq 2160$.

\begin{lemma}\label{lm52}
If $e=4$ and $f\geq 5$, then $5$ cannot divide $N$.
\end{lemma}
\begin{proof}
Assume that $p_1=5$ divides $N$.
If $e_1\neq 2$, then we must have $f=6$ since otherwise
\begin{equation}
\frac{\sbu(N)}{N}>\frac{\sbu(2^4 3^3 5)}{2^4 3^3 5}=3,
\end{equation}
which is a contradiction.
Noting that $13\mid\sbu(3^6)\mid 3N$, $p_2=13$ must divide $N$, $e_2=2$, $p_3=17$ must divide $N$ and $e_3=2$
since otherwise
\begin{equation}
\frac{\sbu(N)}{N}\geq \frac{\sbu(2^4 3^6 5\times 13)}{2^4 3^6 5\times 13}=\frac{287}{90}>3
\end{equation}
or
\begin{equation}
\frac{\sbu(N)}{N}\geq \frac{\sbu(2^4 3^6 5\times 13^2 17)}{2^4 3^6 5\times 13^2 17}=\frac{41}{13}>3,
\end{equation}
a contradiction again.
Since $5^2\mid \sbu(13^2 17^2)$ and we have assumed that $e_1\neq 2$, we must have $e_1\geq 3$ and
\begin{equation}
\frac{\sbu(N)}{N}\geq \frac{\sbu(2^4 3^6 5^4 13^2)}{2^4 3^6 5^4 13^2}=\frac{4879}{1625}>3,
\end{equation}
which is impossible.
Thus, we must have $e_1=2$ and $p_2=13$ must divide $N$ since $13\mid\sbu(5^2)\mid 3N$.

If $f\geq 9$, then, with the aid of Lemma \ref{lm24} (III),
we see that $\sbu(3^f)$ must have two odd prime factors $p_3$ and $p_4$ in addition to $5$ or $13$,
which is incompatible with $e=4$ by Lemma \ref{a}.

If $f=8$, then $61\mid\sbu(3^f)$ and $p_3=61$ must divide $N$.
Noting that $2^2\mid (3^5+1)\mid \sbu(3^8)$, we must have $2^5\mid \sbu(3^8 5^2 13^{e_2} 61^{e_3})$,
which is a contradiction.
If $f=7$, then we have a similar contradiction from $p_3=41\mid(3^4+1)\mid\sbu(3^f)$.
If $f=6$, then $p_3=41$ must divide $N$ again.
By Lemma \ref{lm24} (V), $\sbu(41^{e_3})$ has an odd prime factor $p_4$ other than $3$, $5$, or $p_3$
and, we must have $2^5\mid \sbu(3^6 5^2 13^{e_2} 41^{e_3} p_4^{e_4})$, which is a contradiction.

If $f=5$, then $7\mid\sbu(3^5)$ and $p_3=7$ must divide $N$.
However, we must have $2^5\mid \sbu(3^5 5^2 13^{e_2} 5^{e_3})\mid 3N$ from Lemma \ref{a}, contrary to $e=4$.
Thus we conclude that $5$ cannot divide $N$.
\end{proof}

Now we prove that we can never have $\sbu(N)=3N$ if $N\neq 2160$ and $e=4$.
By Lemmas \ref{lm51} and \ref{lm52}, we must have $f\geq 5$ and $5$ can never divide $N$.

If $f\geq 18$ or $f=13, 14$, then, By Lemma \ref{lm24} (IV), $\sbu(3^f)$ has at least two odd prime factors $p_1, p_2\geq 127$.
By Lemma \ref{a}, we see that $N$ can have at most one more prime factor $p_3$.
Thus, we must have
\begin{equation}
\frac{\sbu(N)}{N}<\frac{27\times 3\times 7\times 127\times 131}{16\times 2\times 6\times 126\times 130}<3,
\end{equation}
which is a contradiction.

If $15\leq f\leq 17$, then $\sbu(3^f)$ has at least three odd prime factors not less than $31$ and
\begin{equation}
\frac{\sbu(N)}{N}<\frac{27\times 3\times 31\times 37\times 41}{16\times 2\times 30\times 36\times 40}<3,
\end{equation}
which is a contradiction.
If $f=7, 8, 10, 11$, then $5\mid\sbu(3^f)$ and $5$ must divide $N$, which is impossible by Lemma \ref{lm52}.

If $f=6, 12$, then $p_1=13$ and $p_2=41$, $547$ respectively must divide $N$.
By Lemma \ref{lm21}, we see that $\sbu(13^{e_1})$ can have at most one prime factor
other than $2$, $3$, or $p_2$,
which is impossible for $f=6$ and $f=12$ by Lemma \ref{lm24} (IV).
If $f=9$, then, $11\times 61$ must divide $N$ and, noting that $2^2\mid\sbu(3^9)$,
$N$ can have no prime factor except $2$, $3$, $11$, or $61$.
Hence, $\sbu(61^{e_2})$ has no prime factor except $2$, $3$, or $11$, which is impossible
since $61^5+1=2\times 11\times 31\times 1238411$.

Finally, if $f=5$, then $\sbu(3^5)=2^2\times 7\times 13$ and therefore $p_2=7$ and $p_3=13$ must divide $N$.
If $e_3\geq 2$, then $\sbu(13^{e_3})$ has another prime factor $p_4$ in addition to $2$, $3$, and $7$
from Lemma \ref{lm24} (IV) and $2^5\mid\sbu(3^5 7^{e_2} 13^{e_3} p_4^{e_4})\mid 3N$, which is a contradiction.
If $e_3=1$, then we have a similar contradiction that $2^5\mid\sbu(3^5 7)\mid 3N$.

Hence, we can never have $\sbu(N)=3N$ if $N\neq 2160$ and $e=4$.
This completes the proof of Theorem \ref{th1}.

{}
\vskip 12pt
\end{document}